\documentclass[11pt, a4paper, reqno]{amsart}
\usepackage{color}
\usepackage{amsmath}
\usepackage{amssymb}
\usepackage{amsfonts}
\usepackage{relsize}
\usepackage[colorlinks, linkcolor=red, citecolor=blue, urlcolor=blue, pagebackref, hypertexnames=false]{hyperref}
\usepackage[hyperpageref]{backref}
\usepackage[lmargin=3cm,rmargin=3cm,tmargin=3cm,bmargin=3cm]{geometry}
\def\R{\mathbb{R}}

\numberwithin{equation}{section}

\newcommand{\EP}{\exp L^p}

\newtheorem{theorem}{Theorem}[section]
\newtheorem{lemma}[theorem]{Lemma}
\newtheorem{proposition}[theorem]{Proposition}
\newtheorem{corollary}[theorem]{Corollary}

\newtheorem{definition}[theorem]{Definition}

\newtheorem{remark}[theorem]{Remark}

\author[M. Majdoub and S. Tayachi]{Mohamed Majdoub and Slim Tayachi}
\address[M. Majdoub]{Department of Mathematics, College of Science\\ Imam Abdulrahman Bin Faisal University,
P.O. Box 1982, Dammam, Saudi Arabia}
\address{Basic \& Applied Scientific Research Center\\ Imam Abdulrahman Bin Faisal University, P.O. Box 1982, 31441, Dammam, Saudi Arabia}
\email{\sl mmajdoub@iau.edu.sa \;\; \& \;\;\; med.majdoub@gmail.com}

\address[S. Tayachi]{Universit\'e de Tunis El Manar, Facult\'e des Sciences de Tunis, D\'epartement de Math\'ematiques, Laboratoire \'equations aux d\'eriv\'ees partielles (LR03ES04), 2092 Tunis, Tunisie} \email{\sl slim.tayachi@fst.rnu.tn\;\; \& \;\;  slimtayachi@gmail.com}


\title[The heat equation with exponential nonlinearity]{Global existence and decay estimates for the heat equation with exponential nonlinearity}

\def\R{\mathbb{R}}

\numberwithin{equation}{section}

\begin{document}

\maketitle

\begin{abstract} In this paper we consider the initial value {problem $\partial_{t} u- \Delta u=f(u),$ $u(0)=u_0\in \exp L^p(\R^N),$} where $p>1$ and $f : \R\to\R$ having an exponential growth at infinity with $f(0)=0.$  Under smallness condition on the initial data and for nonlinearity $f$ {such that  $|f(u)|\sim \mbox{e}^{|u|^q}$ as $|u|\to \infty$,} $|f(u)|\sim |u|^{m}$ as $u\to 0,$ $0<q\leq p\leq\,m,\;{N(m-1)\over 2}\geq p>1$, we show that the solution is global.  Moreover, we obtain decay estimates in Lebesgue spaces  for large time which depend on $m.$
\end{abstract}

\section{Introduction}

In this paper we study  the Cauchy problem:
\begin{equation}\label{1.1}
\left\{\begin{array}{cc}
\partial_{t} u- \Delta u=f(u),   \\
u(0)=u_0\in \exp L^p(\R^N),
\end{array}
\right.
\end{equation}
where $p>1$ and $f : \R\to\R$ having an exponential growth at infinity with $f(0)=0.$ Our  interest is the study of the global existence and the decay estimate of solutions. In particular, we aim to complement the cases  noted in \cite[Remark 1.7(i), p. 2382]{MT-ICM2018}.

The initial value problem \eqref{1.1} has attracted considerable attention in the mathematical community and the well-posedness theory in the Lebesgue spaces, especially for polynomial type nonlinearities, has been extensively studied. It is known that for polynomial nonlinearity one can always find a Lebesgue space $L^q,\; q<\infty$ for which \eqref{1.1} is locally well-posed. See for instance \cite{W, Indiana}. As pointed out in \cite{MT-ICM2018}, by analogy with the Lebesgue spaces, which are well-adapted to  the heat equations  with power  nonlinearities (\cite{Indiana}), we are motivated to consider the Orlicz spaces, in order to study  heat equations with power-exponential nonlinearities.

 For the particular case where $f(u)\sim {\rm e}^{|u|^2},\; u$ large, well-posedness  results are proved in the Orlicz space $\exp L^2(\R^N).$ See \cite{IJMS, Ioku, IRT, RT}. It is also shown that if $f(u)\sim {\rm e}^{|u|^s},\, s>2,\; u$ large then the  existence is no longer guaranteed and in fact there is nonexistence in the Orlicz space  $\exp L^2(\R^N).$  See \cite{IRT}. {Recently the authors of \cite{Fu-Io2018} obtained a sufficient condition on a class of  initial data for time local existence for \eqref{1.1}  with general nonlinearity $f(u)$. In particular, initial data are assumed to be nonnegative and an exponential nonlinearity is discussed as an application. Global existence and decay estimates are also established for the nonlinear heat equation with $f(u)\sim {\rm e}^{|u|^2},\; u$ large. See \cite{Ioku, MOT, FKRT}.

{The Orlicz space $\exp L^p(\R^N)$} is  defined as follows
\begin{equation*}
    \exp L^p(\R^N)=\bigg\{\,u\in L^1_{loc}(\R^N);\;\int_{\R^N}\Big({\rm e}^{|u(x)|^p\over\lambda^p}-1\Big)\,dx<\infty ,\,\;\mbox{for some}\,\; \lambda>0\, \bigg\},
\end{equation*}
endowed with the Luxemburg norm
\begin{equation*}
    \|u\|_{\exp L^p(\R^N)}:=\inf\biggr\{\,\lambda>0;\,\,\,\,\int_{\R^N} \Big({\rm e}^{|u(x)|^p\over\lambda^p}-1\Big)\,dx\leq1\,\biggl\}.
\end{equation*}

As is a standard practice, we study \eqref{1.1} via the associated integral equation:
\begin{equation}
\label{integral}
u(t)= {\rm e}^{t\Delta}u_{0}+\int_{0}^{t}{\rm e}^{(t-s)\Delta}\,f(u(s))\,ds,
\end{equation}
where ${\rm e}^{t\Delta}$ is the linear heat semi-group.

In the sequel, we use the following definition of  weak-mild  solutions to Cauchy problem \eqref{1.1}.
\begin{definition}[{Weak-mild solution}]
\label{weakmild}
We say that $u\in L^{\infty}(0,T;\, \exp L^p(\R^N))$  is a weak-mild solution of
the Cauchy problem \eqref{1.1} if $u$ satisfies the associated integral equation \eqref{integral} in $\exp L^p(\R^N)$ for almost all $t\in (0,T)$ and $u(t)\rightarrow u_0$ in the weak$^*$ topology as $t\searrow 0.$
\end{definition}

{The decay estimates depend} on the behavior of the nonlinearity $f(u)$ near $u=0.$ The following behavior near $0$ will be allowed
$|f(u)|\sim |u|^m,$
where ${N(m-1)\over 2}\geq p.$  More precisely, we suppose that the nonlinearity $f$ satisfies
\begin{equation}\label{1.4}
f(0)=0,\qquad|f(u)-f(v)|\leq C\left|u-v\right|\bigg(|u|^{m-1}{\rm e}^{\lambda |u|^p}+|v|^{m-1}{\rm e}^{\lambda |v|^p}\bigg),
\end{equation}
where  ${N(m-1)\over 2}\geq p>1$, $C>0,$ and $\lambda>0$ are constants. Our aim is to obtain global existence to the Cauchy problem \eqref{1.1} for small initial data in $\exp L^p(\R^N)$. We denote the norm in the Lebesgue space $L^r(\R^N),\; 1\leq r\leq \infty$ by $\|\cdot\|_r.$    We have obtained the following.
\begin{theorem}
\label{GE} Let $N\geq 1,\; m\geq p>1,\; {N(m-1)\over 2}\geq p.$ Assume that the nonlinearity $f$ satisfies \eqref{1.4}. Then, there exists a positive constant $\varepsilon>0$ such that {for} every initial data $u_0\in \exp L^p(\R^N)$
with $\|u_0\|_{\exp L^p(\R^N)} \leqslant\varepsilon,$ there exists a weak-mild solution $u\in L^{\infty}(0,\infty;\exp L^p(\R^N))$
of the Cauchy problem \eqref{1.1} satisfying
\begin{equation} \label{1.8}
    \lim_{t\longrightarrow0}\|u(t)-{\rm e}^{t\Delta}u_{0}\|_{\exp L^p(\R^N)}=0.
\end{equation}
Moreover, we have
\begin{equation}
\label{Linfinibihar}
\|u(t)\|_a \leq\,C\,t^{-\left({1\over m-1}-{N\over 2a}\right)},\quad\,t>0,
\end{equation}
where  $a$ verifies:
\begin{itemize}
\item[(i)] If ${N\over 2}>{p\over p-1}$ then  ${N\over 2}(m-1)<a<{N\over 2}(m-1)\frac{1}{(2-m)_+}\,.$
\item[(ii)] If ${N\over 2}={p\over p-1}$ then ${N\over 2}(m-1)<a<{N\over 2}(m-1)\frac{1}{(2-m)_+}\,.$
\item[(iii)] If ${N\over 2}<{p\over p-1}$ and $(2-m)_+<{N(p-1)\over 2p}$ then ${p\over p-1}(m-1)<a<{N\over 2}(m-1)\frac{1}{(2-m)_+}\,.$
\end{itemize}
In particular, {for all the above  cases,} $0<{1\over m-1}-{N\over 2a}<1.$
\end{theorem}
\begin{remark}{\rm We do not know if the estimate \eqref{Linfinibihar} holds for  more large intervals of $a.$}
\end{remark}
\begin{remark}{\rm The case (i), and only this case, was proved in \cite{MT-ICM2018} but with a supplementary condition on $a$ which we remove here.}
\end{remark}
\begin{remark}{\rm
The assumption $p>1$ is needed in Corollary \ref{lemme estimates} and Corollary \ref{forn=8} below.}\end{remark}
\begin{remark}
{\rm From \cite[Proposition 3.1, p. 88]{Indiana} or \cite[Lemma 7, p. 280]{BC}, we know that the smoothing estimate \eqref{x} below holds for smooth bounded domains and $\Delta=\Delta_D$ with homogeneous Dirichlet boundary conditions. Hence, all linear estimates needed in our proof (see Section 3 below) still holds for smooth bounded domains. This means that we can replace $\R^N$ by any smooth bounded domain $\Omega$ and obtain the same results.}
\end{remark}
The assumption for the nonlinearity covers the cases
$$
f(u)=\pm |u|^{m-1}u{\rm e}^{|u|^p}, \quad m\geq 1+{2p\over N}.
$$
The global existence part of Theorem \ref{GE} is known for $p=2$ (see \cite{Ioku}). The estimate \eqref{Linfinibihar} was obtained in \cite{Ioku} for $p=2$ and $m=1+{4\over N}.$ This is improved in \cite{MOT} for $p=2$ and any $m\geq 1+{4\over N}.$ The fact  that estimate \eqref{Linfinibihar} depends on the smallest power of the nonlinearity $f(u)$ is known in \cite{STW} but only for nonlinearities having polynomial growth. An essential part of our proof consists in a careful combinations of the Taylor expansion and the H\"older inequality, which are firstly developed in \cite{NO1998} for a critical Schr\"odinger equation.

The rest of this paper is organized as follows. In the next section, we recall some basic facts and useful tools about Orlicz spaces.  Section 3 is devoted to some crucial estimates on the linear heat semi-group. {In Section 4 we give the proof of Theorem \ref{GE}. The proof uses similar argument to that in \cite{MT-ICM2018,MOT,Ioku,CW}.} In all this paper, $C$ will be a positive constant which  may have different values at different places. Also, $L^r(\R^N)$, $\exp L^r(\R^N)$, will be written respectively $L^r$ and $\exp L^r$.

\section{Preliminaries}
In this section we recall the definition of the  Orlicz spaces and some
related basic facts. For a complete presentation and more details, we refer the reader to
\cite{Adams,RR,BiOr, HH2019, RR1}. We also give some preliminaries estimates.
\begin{definition}\label{deforl}\quad\\
Let $\phi : [0,\infty)\to[0,\infty)$ be a convex increasing function such that
$$
\phi(0)=0=\lim_{s\to 0^+}\,\phi(s),\quad
\lim_{s\to\infty}\,\phi(s)=\infty.
$$
We say that a  function $u\in L^1_{loc}(\R^N)$ belongs to
$L^\phi$ if there exists $\lambda>0$ such that
$$
\int_{\R^N}\,\phi\left(\frac{|u(x)|}{\lambda}\right)\,dx<\infty.
$$
We denote then
\begin{equation}
\label{Luxemb}
\|u\|_{L^\phi}=\inf\,\left\{\,\lambda>0;\quad\int_{\R^N}\,\phi\left(\frac{|u(x)|}{\lambda}\right)\,dx\leq
1\,\right\}.
\end{equation}
\end{definition}
It is known that $\left(L^\phi,\|\cdot\|_{L^\phi}\right)$ is a Banach space. Note that, if $\phi(s)=s^p,\, 1\leq
p<\infty$,  then $L^\phi$ is nothing else than the  Lebesgue space $L^p$. Moreover, for $u\in L^\phi$ with $K:=\|u\|_{L^\phi}>0$,
we have
$$\left\{\,\lambda>0;\quad\int_{\R^N}\,\phi\left(\frac{|u(x)|}{\lambda}\right)\,dx\leq
1\,\right\}=[K, \infty[\,. $$
In particular
\begin{equation}
\label{med1}
\int_{\R^N}\,\phi\left(\frac{|u(x)|}{\|u\|_{L^\phi}}\right)\,dx\leq 1.
\end{equation}
\begin{remark}
{\rm We may replace in \eqref{Luxemb} the constant $1$ by any positive constant. This change
the norm $\|\cdot\|_{L^\phi}$ to an equivalent norm.}
\end{remark}

We also recall the following well known properties (see \cite[pp. 56-57 and p. 83]{RR1} and \cite[Lemma 3.7.7]{HH2019}).
\begin{proposition} \label{fatou} We have
\begin{itemize}
\item[(i)] $L^1\cap L^\infty\subset L^\phi\subset L^1+L^\infty$.
\item[(ii)] {\it Lower semi-continuity}:
$$
u_n\to u\quad\mbox{a.e.}\quad\Longrightarrow\quad\|u\|_{L^\phi}\leq
\liminf_{n\to\infty}\|u_n\|_{L^\phi}.
$$
\item[(iii)] {\it Monotonicity}:
$$
|u|\leq|v|\quad\mbox{
a.e.}\quad\Longrightarrow\quad\|u\|_{L^\phi}\leq\|v\|_{L^\phi}.
$$
\item[(iv)] {\it Strong Fatou property}:
$$
0\leq u_n\nearrow u\quad\mbox{
a.e.}\quad\Longrightarrow\quad\|u_n\|_{L^\phi}\nearrow\|u\|_{L^\phi}.
$$
\item[(v)] {\it Strong and modular convergence}:
$$
u_n\to u \;\;\mbox{in}\;\; L^\phi \;\;\Longrightarrow\;\; \int_{\R^N}\phi(u_n-u) dx\to 0 .
$$
\end{itemize}
\end{proposition}
When $\phi(s)={\rm e}^{s^p}-1$, we denote the space $L^\phi$ by $\exp L^p$.


 The following Lemma summarize the relationship between Orlicz and Lebesgue spaces.
\begin{lemma}(\cite{RT, MT-ICM2018})\label{Orl-Leb}   We have
\begin{itemize}
\item[(i)] $\EP\not\hookrightarrow L^\infty,$ \; $p\geq 1$.
\item[(ii)] $\EP \not\hookrightarrow L^r$,\;\; for all\;\;$1\leq r<p,$ \; $p>1$.
\item[(iii)] $L^q\cap L^\infty\hookrightarrow \EP$,\;\; for all\;\; $1\leq q\leq p$. More precisely
\end{itemize}
\begin{equation}
\label{embed}
\|u\|_{\exp L^p}\leq \frac{1}{\left(\log 2\right)^{{1\over p}}}\biggr(\|u\|_{q}+\|u\|_{\infty}\biggl)\,.
\end{equation}
\end{lemma}
 We have the embedding: $\exp L^p \hookrightarrow L^{q}$ for every $1\leq p\leq q$. More precisely:
\begin{lemma}(\cite{RT, MT-ICM2018})\label{sarah55} For every $1\leq p\leq q<\infty,$ we have
\begin{equation}\label{2.1}
    \|u\|_{q} \leqslant\left(\Gamma\left(\frac{q}{p}+1\right)\right)^\frac{1}{q}\|u\|_{\exp L^p},
  \end{equation}
where $\Gamma(x):=\displaystyle\int_0^{\infty}\tau^{x-1}{\rm e}^{-\tau}\,d\tau, \; x>0.$
\end{lemma}
{For reader's convenience, we give the proof here.}
\begin{proof} We may assume that
$$
K:=\|u\|_{\exp L^p}>0.
$$
From \cite{RT} we know that
\begin{equation}
\label{RT}
{\rm e}^{z}-1\geq \frac{z^\alpha}{\Gamma(\alpha+1)},\quad z\geq 0,\;\;\;\alpha\geq 1.
\end{equation}
Using \eqref{med1} and \eqref{RT}, we infer
$$
1\geq \int_{\R^N}\, \left({\rm e}^{|\frac{u(x)}{K}|^p}-1\right)\,dx\geq \int_{\R^N}\,\frac{|\frac{u(x)}{K}|^{p\alpha}}{\Gamma(\alpha+1)}\,dx.
$$
Hence
$$
K\geq \frac{\|u\|_{p\alpha}}{(\Gamma(\alpha+1))^{\frac{1}{p\alpha}}}.
$$
We conclude by choosing $\alpha=\frac{q}{p}.$
\end{proof}
We recall that the following properties of the functions $\Gamma$ and ${\mathcal{B}}$ given by
$$ {\mathcal{B}}(x,y)=\int_0^1\tau^{x-1}(1-\tau)^{y-1}d\tau,\quad x,\; y>0.$$ We have
\begin{equation}
\label{gamma4}
 {\mathcal{B}}(x,y)=\frac{\Gamma(x+y)}{\Gamma(x)\Gamma(y)},\quad\; x,\; y>0,
\end{equation}

\begin{equation}
\label{gamma1}
\Gamma(x)\geq C>0,\quad x>0,
\end{equation}

\begin{equation}
\label{gamma2}
\Gamma(x+1)\sim \left({x\over \rm{e}}\right)^x\, \sqrt{2\pi x}, \; \mbox{ as }\; x \to \infty,
\end{equation}
and
\begin{equation}
\label{gamma3}
\Gamma(x+1)\leq C x^{x+{1\over 2}},\quad\; x\geq 1.
\end{equation}

The following lemma will be useful in the proof of the global existence.
\begin{lemma}(\cite[Lemma 2.6, p. 2387]{MT-ICM2018})
\label{med}
Let $\lambda>0$, $1\leq p,\; q<\infty$ and $K>0$ such that $\lambda q\,K^p\leq 1$. Assume that
$$
\|u\|_{\exp L^p}\leq K\,.
$$
Then
$$
\|{\rm e}^{\lambda |u|^p}-1\|_q\leq \left(\lambda q\,K^p\right)^{{1\over q}}\,.
$$
\end{lemma}



\section{Key estimates}

In this section we establish some results needed for the proof of the main theorem. We first recall some basic estimates for the linear heat semigroup ${\rm e}^{t\Delta}.$   The solution of the linear heat equation
\begin{equation*}
\left\{\begin{array}{cc}
\partial_{t} u=\Delta u,\; t>0,\; x\in \R^N,\\
u(0,x)=u_{0}(x),
\end{array}
\right.
\end{equation*}
can be written as a convolution:
\begin{equation*}
    u(t,x)=\bigr(G_t\star u_0\bigl)(x):=\bigr({\rm e}^{t\Delta}u_0\bigl)(x),
\end{equation*}
where
\begin{equation*}
   G_t(x):=G(t,x)=\frac{{\rm e}^{- {|x|^2\over 4t}}}{(4\pi t)^{N\over 2}},\; t>0,\; x\in \R^N,
\end{equation*}
is the heat kernel.  We will  use the $L^r-L^\rho$ estimate as stated in the proposition below.
\begin{proposition}\label{LPLQQ}  For all $1\leq r\leq \rho\leq\infty$, we have
\begin{equation}\label{x}
\|{\rm e}^{t\Delta}\varphi\|_{\rho}\leqslant   t^{-\frac{N}{2}(\frac{1}{r}-\frac{1}{\rho})}\|\varphi\|_{r},\qquad\, t>0,\,\,\; \varphi\in L^r.
\end{equation}
\end{proposition}

The following proposition is established in \cite{MT-ICM2018}.

\begin{proposition}(\cite[Proposition 3.2, p. 2392]{MT-ICM2018})\label{pre} Let $1\leqslant q\leqslant p,\,\,1\leqslant r\leqslant\infty.$ Then the following estimates hold:
\begin{enumerate}
  \item [(i)] $\|{\rm e}^{t\Delta}\varphi\|_{\exp L^p}\leqslant  \|\varphi\|_{\exp L^p},\;\; t>0,$\;\; $\varphi\in \exp L^p.$
  \item [(ii)] $\|{\rm e}^{t\Delta}\varphi\|_{\exp L^p}\leqslant  \,t^{-\frac{N}{2q}}\left(\log(t^{-\frac{N}{2}}+1)\right)
    ^{-\frac{1}{p}} \|\varphi\|_{q},\;\;\; t>0,$\;\;\; $\varphi\in L^q.$
  \item [(iii)] $ \|{\rm e}^{t\Delta}\varphi\|_{\exp L^p}\leqslant  \frac{1}{\left(\log 2\right)^{{1\over p}}}\,\left[\,t^{-\frac{N}{2r}}\|\varphi\|_{r}+\|\varphi\|_{q}\right], \;\;\; t>0,\;\;\;$ $\varphi\in L^r\cap L^q.$
\end{enumerate}
\end{proposition}

We also recall the following from   \cite{MT-ICM2018}.
\begin{corollary}(\cite[Corollary 3.3, p. 2394]{MT-ICM2018})
\label{lemme estimates}
Let $p>1,\; N> {2p\over p-1}, \; r> {N\over 2}.$ Then, for every $g\in L^1\cap L^r,$ we have
 \begin{equation*}
    \|{\rm e}^{t\Delta}g\|_{\exp L^p}\leq \kappa(t)\,\|g\|_{L^1\cap L^r},\;\;\;\; t>0,
\end{equation*}
where $\kappa\in L^1(0,\infty)$ is given by
$$\kappa(t)=\frac{1}{\left(\log 2\right)^{{1\over p}}}\min\biggr\{ t^{-\frac{N}{2r}}+1,\; t^{-\frac{N}{2}}\Big(\log(t^{-\frac{N}{2}}+1)\Big)^{-\frac{1}{p}}\biggl\}.$$
\end{corollary}
Here we use $\|g\|_{L^1\cap L^r}=\|g\|_{1}+\|g\|_{r}.$

For $N=2p/(p-1)$ we have a similar result. For this we need to introduce an appropriate Orlicz space. Let $\phi(s):={\rm e}^{s^p}-1-s^p,\; s\geq 0$ and $L^{\phi}$ be the associated Orlicz space with the Luxemburg norm \eqref{Luxemb}. From the definition, we have
\begin{equation}\label{slim}
    C_1\|u\|_{\exp L^p}\leqslant\|u\|_{p}+\|u\|_{L^{\phi}}\leqslant C_2\|u\|_{\exp L^p},
\end{equation}
for some $C_1,\,C_2>0.$
\begin{corollary}
\label{forn=8}
Let $p>1,\;  r> {N\over 2}= {p\over p-1}.$ For every $g\in L^1\cap L^{2p}\cap L^r,$ we have
 \begin{equation}\label{2.8b}
    \|{\rm e}^{t\Delta}g\|_{L^\phi}\leq \zeta(t)\|g\|_{L^1\cap L^{2p}\cap L^r},\;\;\;\; t>0,
\end{equation}
where $\zeta\in L^1(0,\infty)$ is given by
$$
\zeta(t)=\frac{1}{(\log 2)^{1/p}}\,\min\biggr\{1+t^{-\frac{N}{2r}},\; t^{-{p\over p-1}}\Big(\log(t^{-{p\over p-1}}+1)\Big)^{-\frac{1}{2p}}\biggl\}.
$$
\end{corollary}

\begin{proof} We have, using Proposition \ref{LPLQQ},
\begin{eqnarray*}
  \int_{\R^N}\phi\left(\frac{\left|{\rm e}^{t\Delta}g\right|}{\alpha}\right)\,dx &=&
\sum_{k\geq2}\frac{\|{\rm e}^{t\Delta}g\|^{pk}_{{pk}}}{\alpha^{pk}k!} \\
   &\leq& \sum_{k\geq2}\frac{t^{-{N\over 2}(1-\frac{1}{pk}){pk}}\|g\|_{1}^{pk}}{\alpha^{pk}k!}\\
   &=& \sum_{k\geq2}\frac{t^{-{p\over p-1}(1-\frac{1}{pk}){pk}}\|g\|_{1}^{pk}}{\alpha^{pk}k!}\\
   &=& t^{{p\over p-1}}\phi\left(\frac{t^{-{p\over p-1}}\|g\|_{1}}{\alpha}\right) \\
   &\leq& t^{{p\over p-1}}\left(\exp\Big\{\left(\frac{t^{-{p\over p-1}}\|g\|_{1}}{\alpha}\right)^{2p}\Big\}-1\right),
\end{eqnarray*}
where we have used ${\rm e}^s-1-s\leq {\rm e}^{s^2}-1$ for every $s\geq 0$.
Therefore we obtain that
\begin{eqnarray}\label{2.13}
  \|{\rm e}^{t\Delta}g\|_{L^{\phi}} &\leq& \inf\left\{\alpha>0,
 t^{{p\over p-1}}\left(\exp\Big\{\left(\frac{t^{-{p\over p-1}}\|g\|_{1}}{\alpha}\right)^{2p}\Big\}-1\right)\leq1 \right\} \nonumber\\
   &=& \,t^{-{p\over p-1}}\Big(\log\left(t^{-{p\over p-1}}+1\right)\Big)^{-\frac{1}{2p}} \|g\|_{1}.
\end{eqnarray}
On the other hand, from the embedding $L^{2p}\cap L^{\infty}\hookrightarrow L^{\phi}$, we see that
\begin{equation*}
\begin{split}
\|{\rm e}^{t\Delta}g\|_{L^{\phi}}&\leq\,\frac{1}{({\log 2})^{1/p}}\Big[\|{\rm e}^{t\Delta}g\|_{{\infty}}+\|{\rm e}^{t\Delta}g\|_{{2p}}\Big].
\end{split}
\end{equation*}
Using Proposition \ref{LPLQQ}, and let $r>{N\over 2}={p\over p-1}$ we obtain that
\begin{equation}\label{2.14}
\begin{split}
\|{\rm e}^{t\Delta}g\|_{L^{\phi}}&\leq\, \frac{1}{({\log 2})^{1/p}}\, \Big[ t^{-\frac{N}{2r}}\|g\|_{{r}}+ \|g\|_{{2p}}\Big].
\end{split}
\end{equation}
Combining the inequalities \eqref{2.13} and \eqref{2.14}, we have
\begin{equation*}
   \|{\rm e}^{t\Delta}g\|_{L^{\phi}}\leq \zeta(t)\|g\|_{L^1\cap L^{2p}\cap L^r}.
\end{equation*}
Since $\frac{N}{2r}<1$ and ${p\over p-1}-{p\over p-1}{1\over 2p}={2p-1\over 2(p-1)}>1,$ we have  that $\zeta \in L^{1}(0,\infty)$.
\end{proof}

\section{Proof of the Main Result}


In this section we give the proof of Theorem \ref{GE}. We consider the associated integral equation
 \begin{equation}\label{NN9}
   u(t)= {\rm e}^{t\Delta}u_{0}+\int_{0}^{t}{\rm e}^{(t-s)\Delta}(f(u(s)))ds,
\end{equation}
 where  $\|u_0\|_{\EP}\leq \varepsilon$, with  small $\varepsilon>0$ to be fixed later. The  nonlinearity $f$ satisfies $f(0)=0$ and
\begin{equation}\label{NNN9}
|f(u)-f(v)|\leq C \left|u-v\right|\bigg(|u|^{m-1}{\rm e}^{\lambda |u|^p}+|v|^{m-1}{\rm e}^{\lambda |v|^p}\bigg),
\end{equation}
for some constants $C>0$ and $\lambda>0$. Here  $p>1$ and $m$ is larger than $1+{2p\over N}.$  From \eqref{NNN9}, we  deduce that
\begin{equation}
\label{taylorm}
 |f(u)-f(v)|\leq C|u-v|\sum_{k=0}^{\infty} \frac{\lambda^{k}}{k!} \bigg(|u|^{pk+m-1}+|v|^{pk+m-1}\bigg).
\end{equation}
We will perform a fixed point argument on a suitable metric space.  For $M>0$ we  introduce the space
$$Y_M :=\left\{u\in L^\infty(0,\infty, \EP);\;\displaystyle\sup_{t>0}  t^{\sigma}\|u(t)\|_{a}+\|u\|_{L^{\infty}(0,\infty;\EP)}\leq M\right\},$$ where
$$a>{N(m-1)\over 2}\geq p\; \mbox{ and }\; \sigma={1\over m-1}-\frac{N}{2a}={N\over 2}\left({2\over N(m-1)}-{1\over a}\right)>0.$$ It follows by Proposition \ref{fatou} that endowed with the metric $$d(u,v)=\displaystyle\sup_{t>0} \Big(t^{\sigma}\|u(t)-v(t)\|_{a}\Big),$$ $Y_M$ is a complete metric space.
For $u\in Y_M,$ we define $\Phi(u)$ by
\begin{equation}\label{Phi}
    \Phi(u)(t):={\rm e}^{t\Delta}u_{0}+\int_{0}^{t}{\rm e}^{(t-s)\Delta}(f(u(s))) ds.
\end{equation}

By Propositions \ref{pre} (i), \ref{LPLQQ} and Lemma \ref{sarah55}, we have
\begin{equation}
\label{lioum}
\|{\rm e}^{t\Delta}u_{0}\|_{\EP}\leq \|u_{0}\|_{\EP},
\end{equation}
and
\begin{eqnarray}
\nonumber
t^\sigma \|{\rm e}^{t\Delta}u_{0}\|_a&\leq & t^\sigma t^{-{N\over 2}\left({2\over N(m-1)}-{1\over a}\right)}\|u_{0}\|_{{N(m-1)\over 2}}
\\ \label{med3} &= & \|u_{0}\|_{{N(m-1)\over 2}}\leq C\|u_{0}\|_{\EP},
\end{eqnarray}
where we have used $1<p\leq {N(m-1)\over 2}< a.$

To estimate $\displaystyle\int_{0}^{t}{\rm e}^{(t-s)\Delta}(f(u(s)))$, we use the results and estimates of the previous Sections 2 and 3. We treat separately the cases $N > 2p/(p-1),\; N = 2p/(p-1)$ and $N < 2p/(p-1).$ The proof is
done using similar argument  as that in \cite{Ioku, MOT, NO1998}.

\subsection{The case $N>2p/(p-1)$}
We first recall the following lemma established in \cite[Lemma 2.7, p. 2387]{MT-ICM2018}. Here we remove the condition $a>N/2$ which, in fact is not needed.
\begin{lemma}
\label{params}
Let $m\geq p>1$, $a>\frac{N(m-1)}{2}$. Define
$
\sigma=\frac{1}{m-1}-\frac{N}{2a}>0.
$
Assume that
$
N>\frac{2p}{p-1},
$
$
\label{Aa}
a<\frac{N(m-1)}{2}\frac{1}{(2-m)_+}\,.
$
Then, there exist $r,\;\;q,\;\; (\theta_k)_{k=0}^\infty\;\;, (\rho_k)_{k=0}^\infty$ such that
$$
1\leq r\leq a,\quad
q\geq 1\quad \mbox{and}\quad \frac{1}{r}=\frac{1}{a}+\frac{1}{q}\,.
$$
$$
0<\theta_k<1\quad\mbox{and}\quad \frac{1}{q(pk+m-1)}=\frac{\theta_k}{a}+\frac{1-\theta_k}{\rho_k}\,.
$$
$$
 p\leq \rho_k<\infty, \quad
\frac{N}{2}\left(\frac{1}{r}-\frac{1}{a}\right)<1, \quad
 \sigma\Big[ \theta_k(pk+m-1)+1\Big]<1\,.
$$
$$
1-\frac{N}{2}\left(\frac{1}{r}-\frac{1}{a}\right)-\sigma\theta_k(pk+m-1)=0\,.
$$
Moreover,
\begin{equation}
\label{behav1}
\theta_k\longrightarrow 0\quad\mbox{as}\quad k\longrightarrow\infty.
\end{equation}
\begin{equation}
\label{behav2}
\rho_k\longrightarrow \infty\quad\mbox{as}\quad k\longrightarrow\infty.
\end{equation}
\begin{equation}
\label{behav3}
\frac{(pk+m-1)(1-\theta_k)}{p\rho_k}\,(1+\rho_k)\leq k,\;\;\;\;\;k\geq 1.
\end{equation}
\end{lemma}

We now turn to the proof of the theorem in the case $N>2p/(p-1).$ {This case is done in \cite{MT-ICM2018}. For completeness we give the proof here.} Let $u\in Y_M$. Using Proposition \ref{pre} and Corollary \ref{lemme estimates}, we get for $q>N/2$,
\begin{eqnarray*}
\|\Phi(u)(t)\|_{\EP}&\leq&\|{\rm e}^{t\Delta}u_{0}\|_{\EP}+\int_{0}^{t}\left\|{\rm e}^{(t-s)\Delta}
(f(u(s))) \right\|_{\EP}\,ds\\
&\leq&\|{\rm e}^{t\Delta}u_{0}\|_{\EP}+\int_{0}^{t}\kappa(t-s)\bigg(\|f(u(s))\|_{L^1\cap L^q}
\bigg)\,ds\\
&\leq&\|{\rm e}^{t\Delta}u_{0}\|_{\EP}+ \|f(u)\|_{L^{\infty}(0,\infty;(L^1\cap L^q))}\int_{0}^{\infty}\kappa(s)\,ds\\
&\leq& \|{\rm e}^{t\Delta}u_{0}\|_{\EP}+C \|f(u)\|_{L^{\infty}(0,\infty;(L^1\cap L^q))}.
\end{eqnarray*}
Hence,
$$
\|\Phi(u)\|_{L^\infty(0{,\infty; \EP)}}\leq \|u_{0}\|_{\EP}+ C \|f(u)\|_{L^{\infty}(0,\infty;\,L^1\cap L^q)}.
$$
It remains to estimate the nonlinearity $f(u)$ in $L^r$ for $r=1,\,q.$ To this end, let us remark that
\begin{equation}\label{3.4}
|f(u)|
\leq C|u|^m\left({\rm e}^{\lambda |u|^p}-1\right)+C|u|^m.
\end{equation}
By H\"{o}lder's inequality and Lemma \ref{sarah55}, we have for $1\leq r\leq q$ and since $m\geq p$,
\begin{eqnarray}\label{3.6}
\nonumber
\|f(u)\|_{{r}} &\leq & C\|u\|_{{mr}}^m +C\||u|^m({\rm e}^{\lambda |u|^p}-1)\|_{{r}}\\ & \leq & C\|u\|_{{mr}}^m +C\|u\|_{{2mr}}^m\|{\rm e}^{\lambda |u|^p}-1\|_{{2r}} \\ &\leq& \nonumber C\|u\|_{\EP}^m\bigg(\|{\rm e}^{\lambda |u|^p}-1\|_{{2r}}+1\bigg).
\end{eqnarray}
According to Lemma \ref{med},  and the fact that $u\in Y_M$, we have for $2q\lambda M^p\leq 1$,
\begin{equation}
\label{3.8}
\|f(u)\|_{L^\infty(0,\infty; L^{r})}\leq CM^m.
\end{equation}

 Finally, we obtain
\begin{eqnarray*}
\|\Phi(u)\|_{L^\infty(0,\infty, \EP)}&\leq&  \|u_{0}\|_{\EP}+C M^m\\
&\leq& \varepsilon+CM^m.
\end{eqnarray*}

Let  $u,\,v$ be two elements of $Y_M.$ By using \eqref{taylorm} and Proposition \ref{LPLQQ}, we obtain
\begin{eqnarray*}
 t^{\sigma}\|\Phi(u)(t)-\Phi(v)(t)\|_{a}&\leq&
t^{\sigma}\int_{0}^{t}\left\| {\rm e}^{(t-s)\Delta}
(f(u(s))-f(v(s)))\right\|_{a} ds\\&\leq& t^{\sigma}
\int_{0}^{t}(t-s)^{-{N\over 2}(\frac{1}{r}-\frac{1}{a})}\left\|f(u(s))-f(v(s))\right\|_{r}\,ds\\
&&\hspace{-3cm}\leq C\sum_{k=0}^{\infty}\frac{\lambda^{k}}{k!} t^{\sigma}\int_{0}^{t}(t-s)^{-{N\over 2}(\frac{1}{r}-\frac{1}{a})}
\|(u-v)(|u|^{pk+m-1}+|v|^{pk+m-1})\|_{r}ds,
\end{eqnarray*}
where $1\leq r\leq a.$ We use the H\"{o}lder inequality with ${1\over r}={1\over a}+{1\over q}$ to obtain
\begin{eqnarray*}
 t^{\sigma}\|\Phi(u)(t)-\Phi(v)(t)\|_{a}&\leq& C\sum_{k=0}^{\infty}\frac{\lambda^{k}}{k!} t^{\sigma}\int_{0}^{t}(t-s)^{-{N\over 2}(\frac{1}{r}-\frac{1}{a})}
\|u-v\|_{a}\times\\
&&\||u|^{pk+m-1}+|v|^{pk+m-1}\|_{q}ds,\\
&\leq&  C\sum_{k=0}^{\infty}\frac{\lambda^{k}}{k!} t^{\sigma}\int_{0}^{t}(t-s)^{-{N\over 2}(\frac{1}{r}-\frac{1}{a})}
\|u-v\|_{a}\times\\ &&\left(\|u\|^{pk+m-1}_{{q(pk+m-1)}}+\|v\|^{pk+m-1}_{{q(pk+m-1)}}\right)ds.
\end{eqnarray*}

Using interpolation inequality with $\frac{1}{q(pk+m-1)}=\frac{\theta}{a}+\frac{1-\theta}{\rho},\; p\leq \rho<\infty,$  we find that
\begin{equation*}
\begin{split}
 t^{\sigma}\|\Phi(u)(t)-\Phi(v)(t)\|_{a}\,\,\!\!\!\!
&\leq C\sum_{k=0}^{\infty}\frac{\lambda^{k}}{k!}t^{\sigma}\int_0^t\,(t-s)^{-\frac{N}{2}
(\frac{1}{r}-\frac{1}{a})}\|u-v\|_{a}\\
&\hspace{-5.5cm}\times\biggl(\|u\|^{(pk+m-1)\theta}_{{a}}
  \|u\|^{(pk+m-1)(1-\theta)}_{{\rho}}
  +\|v\|^{(pk+m-1)\theta}_{{a}}\|v\|^{(pk+m-1)(1-\theta)}_{{\rho}}\biggr)\,ds.
\end{split}
\end{equation*}
By Lemma \ref{sarah55}, we obtain
\begin{eqnarray}\label{4.16b}
&& t^{\sigma}\|\Phi(u)(t)-\Phi(v)(t)\|_{a}\nonumber\\
&&\qquad\qquad\leq C\sum_{k=0}^{\infty}\frac{\lambda^{k}}{k!}t^{\sigma}
\int_0^t\,(t-s)^{-\frac{N}{2}(\frac{1}{r}-\frac{1}{a})}\|u-v\|_{a}\Gamma\left(\frac{\rho}{p}+1\right)
^{\frac{(pk+m-1)(1-\theta)}{\rho}}\qquad\qquad\nonumber\\
&&\qquad\times\left(\|u\|^{(pk+m-1)\theta}_{{a}}\|u\|^{(pk+m-1)(1-\theta)}_{\EP}+
\|v\|^{(pk+m-1)\theta}_{{a}}\|v\|^{(pk+m-1)(1-\theta)}_{\EP}\right)\,ds.
\end{eqnarray}
Applying the fact that $u,\; v\in\; Y_M $ in  \eqref{4.16b}, we see that
\begin{eqnarray}\label{4.18sb}
&& t^{\sigma}\|\Phi(u)(t)-\Phi(v)(t)\|_{a}\nonumber\\
&&\qquad\leq Cd(u,v)\sum_{k=0}^{\infty}\frac{\lambda^{k}}{k!}\Gamma\left(\frac{\rho}{p}+1\right)^{\frac{(pk+m-1)(1-\theta)}{\rho}} M^{pk+m-1}
\nonumber\\
&&\qquad \qquad\qquad\times t^{\sigma}\bigg(\int_0^t(t-s)^{-\frac{N}{2}(\frac{1}{r}-\frac{1}{a})}s^{-\sigma(1+(pk+m-1)\theta)}\,ds\bigg)\nonumber\\
&&\qquad\leq Cd(u,v)\,\sum_{k=0}^{\infty}\frac{\lambda^{k}}{k!}\Gamma\left(\frac{\rho}{p}+1\right)^{\frac{(pk+m-1)(1-\theta)}{\rho}} M^{pk+m-1}
\qquad\qquad\qquad\qquad\qquad\nonumber\\
&&\qquad\qquad\qquad\times {\mathcal{B}}\left(1-\frac{N}{2}\left(\frac{1}{r}-\frac{1}{a}\right),1-\sigma\big(1+(pk+m-1)\theta\big)\right),
\end{eqnarray}
where the parameters $a,\,q,\,r,\,\theta=\theta_k,\,\rho=\rho_k$  are given by Lemma \ref{params}.  Remark that $\theta_k$ satisfies
$$
0<\theta_k<\frac{1}{pk+m-1}\min\Big(m-1, \frac{1-\sigma}{\sigma}\Big)\,.
$$
For these parameters, using \eqref{gamma4} and \eqref{gamma1}, we obtain that
\begin{equation}\label{4.20b}
{\mathcal{B}}\left(1-\frac{N}{2}\left(\frac{1}{r}-\frac{1}{a}\right),1-\sigma\big(1+(pk+m-1)\theta\big)\right)\leq C,
\end{equation}
where $C>0$ is a constant independent on $k$. Moreover, using \eqref{behav1}-\eqref{behav2}-\eqref{behav3} together with \eqref{gamma3} and \eqref{gamma2} gives
\begin{equation}\label{4.21b}
   \Gamma\left(\frac{\rho_k}{p}+1\right)^{\frac{(pk+m-1)(1-\theta_k)}{\rho_k}} \leq C^k k !.
\end{equation}
 Combining \eqref{4.18sb}, \eqref{4.20b} and \eqref{4.21b} we get
$$
 t^{\sigma}\|\Phi(u)(t)-\Phi(v)(t)\|_{a}\leq C d(u,v) \sum_{k=0}^{\infty} \,{(C\lambda)^k} M^{pk+m-1}.$$
Hence, we get for $M$ small,
\begin{equation*}\label{4.18sss}
d\left(\Phi(u), \Phi(v)\right)\leq C M^{m-1} d(u,v).
\end{equation*}

The above estimates show that $\Phi : Y_M \to Y_M $ is a contraction mapping for  $\varepsilon$ and $ M$ sufficiently small.  By Banach's fixed point theorem, we thus obtain the existence of a unique
 $u$ in  $Y_M$ with $\Phi(u)=u.$ By \eqref{Phi}, $u$ solves the integral equation \eqref{NN9} with $f$ satisfying \eqref{NNN9}. The estimate \eqref{Linfinibihar} follows from $u\in Y_M.$ This terminates the proof of the existence of a  global solution to \eqref{NN9} for $N>2p/(p-1)$.



\subsection{The case of $N<2p/(p-1)$}

According to \eqref{med3} and \eqref{lioum},  it remains to establish the following two inequalities
\begin{equation}
\label{estim1}
\left\|\int_{0}^{t}{\rm e}^{(t-s)\Delta}(f(u(s))) ds\right\|_{L^{\infty}(0,\infty;\,\EP)}\leq C_1(M),
\end{equation}
and
 \begin{equation}
\label{estim2}
\sup_{t>0}t^{\sigma}\left\|\int_0^t{\rm e}^{(t-s)\Delta}\,(f(u)-f(v))\,ds\right\|_{a}\leq C_2(M)\sup_{s>0}\,\left(s^{\sigma}\|u(s)-v(s)\|_{a}\right),
\end{equation}
 where $u,\;v\in Y_M$ and with $C_1$ and $C_2$ are small when $M$ is small.\\

\noindent{\sc Estimate \eqref{estim1}}.  We have
\begin{equation}\label{log}
    \left(\log\left((t-s)^{-N/2}+1\right)\right)^{-\frac{1}{p}}\leq {2^{1/p}}(t-s)^{\frac{N}{2p}}\quad \mbox{for}\quad 0\leq s< t-\eta^{-\frac{2}{N}},
\end{equation}
{where $\eta=\inf\{\, z\geq 1;\;\; z> 2\log(1+z)\,\}.$   Therefore,} using Proposition \ref{pre} Part (iii), we have, for $r>N/2$ and $0<t\leq \eta^{-2/N}$,
\begin{eqnarray*}
\left\|\int_0^t{\rm e}^{(t-s)\Delta}\,f(u(s))\,ds\right\|_{\EP}&\leq&C\int_0^t\Big((t-s)^{-\frac{N}{2r}}+1\Big)\|f(u(s))\|_{{L^r\cap L^1}}\,ds\\
&\leq&C\sup_{t>0}\|f(u(t))\|_{{L^r\cap L^1}}.
\end{eqnarray*}
For $t\geq \eta^{-2/N}$ and $1\leq q\leq p$, we write
\begin{eqnarray}\label{(4.31)}
&&\qquad\qquad\left\|\int_0^t{\rm e}^{(t-s)\Delta}\,f(u(s))\,ds\right\|_{\EP}\leq\int_0^t\left\|{\rm e}^{(t-s)\Delta}\,f(u(s))\right\|_{\EP}\,ds\nonumber\\
&&\qquad\leq\int_0^{t-\eta^{-\frac{2}{N}}}\,\left\|{\rm e}^{(t-s)\Delta}\,f(u(s))\right\|_{\EP}\,ds+\int_{t-\eta^{-\frac{2}{N}}}^t\,\left\|{\rm e}^{(t-s)\Delta}\,f(u(s))\right\|_{\EP}\,ds\nonumber\\
&&\qquad\leq  \int_0^{t-\eta^{-\frac{2}{N}}}(t-s)^{-\frac{N}{2q}}(\log((t-s)^{-\frac{N}{2}}+1))^{-\frac{1}{p}}\|f(u(s))\|_{q}\,ds
\qquad\qquad\qquad \nonumber\\ \nonumber
&&\qquad \qquad+
\int_{t-\eta^{-\frac{2}{N}}}^t\Big((t-s)^{-\frac{N}{2r}}+1\Big)\|f(u(s))\|_{{L^r\cap L^1}}\,ds\\ \nonumber
&&\qquad\leq C\int_0^t(t-s)^{-\frac{N}{2q}+{\frac{N}{2p}}}\|f(u(s))\|_{q}\,ds+C\sup_{t>0}\|f(u(t))\|_{{L^r\cap L^1}}\nonumber= \textbf{I}+\textbf{J}.
\end{eqnarray}
We first estimate $\textbf{I}$. By \eqref{taylorm} and the fact that $f(0)=0$,  we have
\begin{equation*}
\textbf{I}=\int_0^t(t-s)^{-\frac{N}{2q}+{\frac{N}{2p}}}\|f(u(s))\|_{q}\,ds\leq C\sum_{k=0}^{\infty}\frac{\lambda^k}{k!}\int_0^t(t-s)^{-\frac{N}{2q}+\frac{N}{2p}}\|u\|_{{(pk+m)q}}^{pk+m}\,\,ds.
\end{equation*}
Using interpolation inequality and Lemma \ref{sarah55}, we get
\begin{eqnarray*}
\textbf{I}&\leq&C\sum_{k=0}^{\infty}\frac{\lambda^k}{k!}\int_0^t(t-s)^{-\frac{N}{2q}+\frac{N}{2p}}\|u\|_{a}^{(pk+m)
\theta}\|u\|_{\rho}^{(pk+m)(1-\theta)}\,ds\nonumber\\&\leq&
C\sum_{k=0}^{\infty}\frac{\lambda^k}{k!}\int_0^t(t-s)^{-\frac{N}{2q}+\frac{N}{2p}}\|u\|_{a}^{(pk+m)
\theta}\Gamma\left(\frac{\rho}{p}+1\right)
^{\frac{(pk+m)(1-\theta)}{\rho}}\|u\|_{\EP}^{(pk+m)(1-\theta)}\,ds\nonumber\\
&\leq&C\sum_{k=0}^{\infty}\frac{\lambda^k}{k!}\Gamma\left(\frac{\rho}{p}+1\right)
^{\frac{(pk+m)(1-\theta)}{\rho}}M^{pk+m}\int_0^t(t-s)^{-\frac{N}{2q}+\frac{N}{2p}}s^{-(pk+m)\theta\sigma}\,ds
\nonumber\\
&\leq&C\sum_{k=0}^{\infty}\frac{\lambda^k}{k!}\Gamma\left(\frac{\rho}{p}+1\right)
^{\frac{(pk+m)(1-\theta)}{\rho}}M^{pk+m}t^{1-\frac{N}{2q}+\frac{N}{2p}-(pk+m)\theta\sigma}\\
&&\times {\mathcal{B}}\left(1-\frac{N}{2q}+\frac{N}{2p},1-(pk+m)\theta\sigma\right),
\end{eqnarray*}
where ${\mathcal{B}}$ is the beta function and $\rho,\,\theta,\,q$  satisfy, for all $k,$
\begin{eqnarray*}
a>{(m-1)p\over p-1},\quad 0\leq \theta=\theta_k\leq1,\quad  1\leq q\leq p,\quad \frac{N}{2q}-{\frac{N}{2p}}<1,\quad && (pk+m)\theta\sigma<1, \\
 \quad 1-\frac{N}{2q}+\frac{N}{2p}-(pk+m)\theta\sigma=0, \quad \frac{1}{(pk+m)q}=\frac{\theta}{a}+\frac{1-\theta}{\rho},\quad &&\;  p\leq \rho=\rho_k<\infty.
\end{eqnarray*}
For any $a>{(m-1)p\over p-1}$, one can choose $$\frac{1-{N(p-1)\over 2p}}{(pk+m)\sigma}<\theta_k<\frac{m-1}{pk+m}.$$ It is obvious that for such $\theta_k$, there exist $q,\,\rho$ such that the above conditions are satisfied. Note that $1-{N(p-1)\over 2p}>0$ in the present case, this gives the supplementary condition on $a.$

Arguing as above, we obtain
\begin{equation}\label{(4.34)}
  {\mathcal{B}}\left(1-\frac{N}{2q}+\frac{N}{2p},1-(pk+m)\theta\sigma\right)={1\over \Gamma\left(1-\frac{N}{2q}+\frac{N}{2p}\right) \Gamma\Big(1-(pk+m)\theta\sigma\Big)}\leq C,
\end{equation}
and
\begin{equation}\label{(4.35)}
  \Gamma\left(\frac{\rho_k}{p}+1\right)^{\frac{(pk+m)(1-\theta_k)}{\rho_k}} \leq C^k k!,
\end{equation}

Combining  \eqref{(4.34)} and \eqref{(4.35)}, we have, for small M,
\begin{equation}\label{(4.36)}
  \textbf{I}\leq C\,M^m.
\end{equation}

To estimate the term $\textbf{J}$, we write
\begin{equation*}
    \|f(u)\|_{\tau}\leq C\||u|^m {\rm e}^{\lambda |u|^p}\|_{\tau}\leq C\||u|^m({\rm e}^{\lambda |u|^p}+1-1)\|_{\tau},
\end{equation*}
with $\tau=r$ or $\tau=1.$ By H\"older inequality, we obtain
\begin{eqnarray*}
   \|f(u)\|_{\tau}&\leq& C\|u\|^m_{{2m\tau}}\|{\rm e}^{\lambda |u|^p}-1\|_{{2\tau}}+\|u(t)\|^m_{{m\tau}}\\
&\leq&C\|u\|^m_{\EP}\left(\|{\rm e}^{\lambda |u|^p}-1\|_{{2\tau}}+1\right),
\end{eqnarray*}
where we have used $m\tau>Nm/2>N(m-1)/2\geq p$ and $m\geq p.$ Now, by Lemma \ref{med}, for $2\tau\lambda M^p\leq 1,$ we have
\begin{equation*}
\|{\rm e}^{\lambda |u|^p}-1)\|_{{2\tau}}\leq (2\tau\lambda M^p)^{\frac{1}{2\tau}}\leq 1.
\end{equation*}
Then we conclude that, for $u\in Y_M,$
\begin{equation*}
   \textbf{J}=C\sup_{t>0}\|f(u(t))\|_{{L^r\cap L^{1}}}\leq C M^m.
\end{equation*}
Finally, we obtain
\begin{eqnarray}\label{(4.991)}
\left\|\int_0^t{\rm e}^{(t-s)\Delta}\,(f(u))\,ds\right\|_{\EP
}&\leq&CM^m.
\end{eqnarray}

{\sc Estimate \eqref{estim2}}. By \eqref{taylorm} and Proposition \ref{LPLQQ}, we have
\begin{equation*}
\begin{split}
&t^{\sigma}\left\|\int_0^t{\rm e}^{(t-s)\Delta}\,(f(u)-f(v))\,ds\right\|_{a}\\
&\qquad\qquad\qquad\leq C\, \sum_{k=0}^{\infty}\frac{\lambda^k}{k!}t^{\sigma}\int_0^t
(t-s)^{-\frac{N}{2}(\frac{1}{r}-\frac{1}{a})}\|(u-v)(|u|^{pk+m-1}+|v|^{pk+m-1})\|_{r}\,ds.\qquad\qquad
\end{split}
\end{equation*}
Applying the H\"older's inequality, we obtain
\begin{eqnarray*}
&& t^{\sigma}\left\|\int_0^t{\rm e}^{(t-s)\Delta}\,(f(u)-f(v))\,ds\right\|_{a}\\
&&\quad\leq C\, \sum_{k=0}^{\infty}\frac{\lambda^k}{k!}t^{\sigma}\int_0^t(t-s)^{-\frac{N}{2}(\frac{1}{r}-\frac{1}{a})}\|u-v\|_{a} \|(|u|^{pk+m-1}+|v|^{pk+m-1})\|_{q}\,ds\\
&&\quad\leq C\, \sum_{k=0}^{\infty}\frac{\lambda^k}{k!}t^{\sigma}\int_0^t(t-s)^{-\frac{N}{2}
(\frac{1}{r}-\frac{1}{a})}\|u-v\|_{a}\left(\|u\|^{pk+m-1}_{{q(pk+m-1)}}+\|v\|^{pk+m-1}_{{q(pk+m-1)}}\right)\,ds.
\end{eqnarray*}
Using interpolation inequality where $\frac{1}{q(pk+m-1)}=\frac{\theta}{a}+\frac{1-\theta}{\rho},\; p\leq \rho<\infty,$  we have
\begin{equation*}
\begin{split}
t^{\sigma}\left\|\int_0^t{\rm e}^{(t-s)\Delta}\,(f(u)-f(v))\,ds\right\|_{a}\,\,\!\!\!\!
&\leq \sum_{k=0}^{\infty}\frac{\lambda^k}{k!}t^{\sigma}\int_0^t(t-s)^{-\frac{N}{2}
(\frac{1}{r}-\frac{1}{a})}\|u-v\|_{a}\\
&\hspace{-3cm}\times(\|u\|^{(pk+m-1)\theta}_{{a}}
  \|u\|^{(pk+m-1)(1-\theta)}_{{\rho}}
  +\|v\|^{(pk+m-1)\theta}_{{a}}\|v\|^{(pk+m-1)(1-\theta)}_{{\rho}})\,ds.
\end{split}
\end{equation*}
By Lemma \ref{sarah55}, we obtain
\begin{eqnarray}\label{4.16}
&&t^{\sigma}\left\|\int_0^t{\rm e}^{(t-s)\Delta}\,(f(u)-f(v))\,ds\right\|_{a}\nonumber\\
&&\qquad\qquad\leq C\sum_{k=0}^{\infty}\frac{\lambda^k}{k!}t^{\sigma}
\int_0^t(t-s)^{-\frac{N}{2}(\frac{1}{r}-\frac{1}{a})}\|u-v\|_{a}\Gamma\left(\frac{\rho}{p}+1\right)
^{\frac{(pk+m-1)(1-\theta)}{\rho}}\qquad\qquad\nonumber\\
&&\quad\times\left(\|u\|^{(pk+m-1)\theta}_{{a}}\|u\|^{(pk+m-1)(1-\theta)}_{\EP}+
\|v\|^{(pk+m-1)\theta}_{{a}}\|v\|^{(pk+m-1)(1-\theta)}_{\EP}\right)\,ds.
\end{eqnarray}
Applying the fact that $u,\; v\in\; Y_M $ in  \eqref{4.16}, we see that
\begin{eqnarray}\label{4.18s}
&&t^{\sigma}\left\|\int_0^t{\rm e}^{(t-s)\Delta}\,(f(u)-f(v))\,ds\right\|_{a}\nonumber\\
&&\qquad\leq Cd(u,v)\sum_{k=0}^{\infty}\frac{\lambda^k}{k!}\Gamma\left(\frac{\rho}{p}+1\right)^{\frac{(pk+m-1)(1-\theta)}{\rho}} M^{pk+m-1}
\nonumber\\
&&\qquad \qquad\qquad\times t^{\sigma}\bigg(\int_0^t(t-s)^{-\frac{N}{2}(\frac{1}{r}-\frac{1}{a})}s^{-\sigma(1+(pk+m-1)\theta)}\,ds\bigg)\nonumber\\
&&\qquad\leq Cd(u,v)\,\sum_{k=0}^{\infty}\frac{\lambda^k}{k!}\Gamma\left(\frac{\rho}{p}+1\right)^{\frac{(pk+m-1)(1-\theta)}{\rho}} M^{pk+m-1}
\qquad\qquad\qquad\qquad\qquad\nonumber\\
&&\qquad\qquad\qquad\times {\mathcal{B}}\left(1-\frac{N}{2}\left(\frac{1}{r}-\frac{1}{a}\right),1-\sigma(1+(pk+m-1)\theta)\right),
\end{eqnarray}
where the exponents $a,\,q,\,r,\,\theta,\,\rho$  satisfy for all $k,$
\begin{eqnarray*}
m<a<{N(m-1)\over 2(2-m)_+}, \quad 1\leq r\leq a, && \frac{N}{2}\left(\frac{1}{r}-\frac{1}{a}\right)<1,\quad \sigma\big(1+(pk+m-1)\theta\big)<1,\nonumber\\
  0\leq\theta=\theta_k\leq1,\quad \frac{1}{r}=\frac{1}{a}+\frac{1}{q}, &&\quad \frac{1}{(pk+m-1)q}=\frac{\theta}{a}+\frac{1-\theta}{\rho},\quad p\leq \rho<\infty,\nonumber\\
  &&\hspace{-3cm}1-\frac{N}{2}\left(\frac{1}{r}-\frac{1}{a}\right)-
  (pk+m-1)\theta\sigma=0.
\end{eqnarray*}
For any $a>m$, one can choose
$$\frac{{N\over 2a}+1-{N\over 2}}{(pk+m-1)\sigma}<\theta_k<\frac{1}{pk+m-1}\min\left(m-1,\,{1-\sigma\over \sigma}\right).$$ It is obvious that for such $\theta_k$, there exist $r,\,q,\,\rho$ such that the above conditions are satisfied.

Using \eqref{gamma4}, \eqref{gamma1} and the fact that $1-\sigma>0$,  we obtain that
\begin{equation}\label{4.20}
 {\mathcal{B}}\left(1-\frac{N}{2}\left(\frac{1}{r}-\frac{1}{a}\right),1-\sigma\big(1+(pk+m-1)\theta\big)\right)\leq C,
\end{equation}
where $C>0$ is a constant independent on $k$. As above, we also have
\begin{equation}\label{4.21}
   \Gamma\left(\frac{\rho_k}{p}+1\right)^{\frac{(pk+m-1)(1-\theta_k)}{\rho_k}} \leq C^k k !.
\end{equation}
 Combining \eqref{4.18s}, \eqref{4.20} and \eqref{4.21} we have
\begin{eqnarray*}
&&t^{\sigma}\left\|\int_0^t{\rm e}^{(t-s)\Delta}\,(f(u)-f(v))\,ds\right\|_{a}\nonumber\\
&&\qquad\qquad\leq C\,d(u,v)\,\sum_{k=0}^{\infty}\frac{\lambda^k}{k!}\Gamma\left(\frac{\rho}{p}+1\right)^{\frac{(pk+m-1)(1-\theta)}{\rho}} M^{pk+m-1}\nonumber\\
&&\qquad\qquad\qquad\times {\mathcal B}\left(1-\frac{N}{2}\left(\frac{1}{r}-\frac{1}{a}\right),1-\sigma(1+(pk+m-1)\theta)\right)\nonumber\\
&&\qquad\qquad\leq C d(u,v) \sum_{k=0}^{\infty}{(C\lambda)^k} M^{pk+m-1}.\nonumber
\end{eqnarray*}
Then, we get (for small $M$)
\begin{equation*}\label{4.18ssss}
t^{\sigma}\left\|\int_0^t{\rm e}^{(t-s)\Delta}\,(f(u)-f(v))\,ds\right\|_{a}\leq C M^{m-1} d(u,v).
\end{equation*}
This together with \eqref{(4.991)} and \eqref{med3} concludes the proof of global existence for dimensions $N<2p/(p-1)$.


\subsection{The case $N=2p/(p-1)$}
Let  $u,\,v$ be two elements of $Y_M.$ By using \eqref{taylorm} and Proposition \ref{LPLQQ}, we obtain
\begin{eqnarray*}
 t^{\sigma}\|\Phi(u)(t)-\Phi(v)(t)\|_{a}&\leq&
t^{\sigma}\int_{0}^{t}\left\| {\rm e}^{(t-s)\Delta}
(f(u(s))-f(v(s)))\right\|_{a} ds\\&\leq&  t^{\sigma}
\int_{0}^{t}(t-s)^{-{N\over 2}(\frac{1}{r}-\frac{1}{a})}\left\|f(u(s))-f(v(s))\right\|_{r}\,ds\\
&&\hspace{-3cm}\leq C\sum_{k=0}^{\infty}\frac{\lambda^{k}}{k!} t^{\sigma}\int_{0}^{t}(t-s)^{-{N\over 2}(\frac{1}{r}-\frac{1}{a})}
\|(u-v)(|u|^{pk+m-1}+|v|^{pk+m-1})\|_{r}ds,
\end{eqnarray*}
where $1\leq r\leq a.$ We use the H\"{o}lder inequality with ${1\over r}={1\over a}+{1\over q}$ to obtain
\begin{eqnarray*}
 t^{\sigma}\|\Phi(u)(t)-\Phi(v)(t)\|_{a}&\leq& C\sum_{k=0}^{\infty}\frac{\lambda^{k}}{k!} t^{\sigma}\int_{0}^{t}(t-s)^{-{N\over 2}(\frac{1}{r}-\frac{1}{a})}
\|u-v\|_{a}\times\\
&&\||u|^{pk+m-1}+|v|^{pk+m-1}\|_{q}ds,\\
&\leq&  C\sum_{k=0}^{\infty}\frac{\lambda^{k}}{k!} t^{\sigma}\int_{0}^{t}(t-s)^{-{N\over 2}(\frac{1}{r}-\frac{1}{a})}
\|u-v\|_{a}\times\\ &&\left(\|u\|^{pk+m-1}_{{q(pk+m-1)}}+\|v\|^{pk+m-1}_{{q(pk+m-1)}}\right)ds.
\end{eqnarray*}
Similar calculations as in the subsection 4.1, give
\begin{equation}\label{4.12}
t^{\sigma}\|\Phi(u)-\Phi(v)\|_{a}\leq  CM^{m-1}\sup_{s>0}\left(s^{\sigma}\|u-v\|_{a}\right)=CM^{m-1} d(u,v).
\end{equation}
{Hence, we need $a$ to satisfy ${N(m-1)\over 2}<a<{N(m-1)\over 2}{1\over (2-m)_+}.$}

We now estimate $\|\Phi(u)\|_{L^{\infty}(0,\infty;\exp L^p)}.$ We have, by  \eqref{lioum} and \eqref{slim},
\begin{equation*}
\begin{split}
\|\Phi(u)\|_{L^{\infty}(0,\infty;\exp L^p)}&\leq \|{\rm e}^{t\Delta}u_{0}\|_{L^{\infty}(0,\infty;\exp L^p)}\\
&\hspace{1cm}+\left\|\int_{0}^{t}{\rm e}^{(t-s)\Delta} (f(u(s))) ds\right\|_{L^{\infty}(0,\infty;\exp L^p)}\\ &\hspace{-3cm}\leq \|u_{0}\|_{\exp L^p}
+\left\|\int_{0}^{t}{\rm e}^{(t-s)\Delta} (f(u(s))) ds\right\|_{L^{\infty}(0,\infty;\exp L^p)}\\ &\hspace{-3cm}\leq \|u_{0}\|_{\exp L^p}+\left\|\displaystyle\int_0^t{\rm e}^{(t-s)\Delta}(f(u(s)))\,ds\right\|_{L^{\infty}(0,\infty;L^{\phi})}\\&\hspace{0.5cm}+\left\|\displaystyle\int_0^t{\rm e}^{(t-s)\Delta}(f(u(s)))\,ds\right\|_{L^{\infty}(0,\infty;L^{p})}.
\end{split}
\end{equation*}

We first estimate $\|\displaystyle\int_{0}^{t}{\rm e}^{(t-s)\Delta} (f(u(s)))\,ds\|_{L^{\infty}(0,\infty;L^{\phi})}$. By the same argument as in the case $N> 2p/(p-1),$  using Corollary \ref{forn=8}, we obtain
\begin{equation}
\label{p}
\begin{split}
\left\|\int_{0}^{t}{\rm e}^{(t-s)\Delta} (f(u(s)))\,ds\right\|_{L^{\infty}(0,\infty;L^{\phi})} &\leq C M^m.
\end{split}
\end{equation}
Second  we estimate $\left\|\displaystyle\int_0^t{\rm e}^{(t-s)\Delta}\,(f(u(s)))\, ds\right\|_{L^{\infty}(0,\infty;L^p)}.$
By using \eqref{taylorm} and Proposition \ref{LPLQQ}, we obtain
\begin{equation*}
\left\|\int_{0}^{t}{\rm e}^{(t-s)\Delta}(f(u(s))) ds\right\|_{p}\leq C\int_{0}^{t}\left\|{\rm e}^{(t-s)\Delta}(f(u(s)))\right\|_{p} ds\leq C\int_{0}^{t}\left\|f(u(s))\right\|_{p} ds.
\end{equation*}
Using similar computations as for the term $\textbf{I},$ in the case $N<2p/(p-1),$ where we take $q=p $ there, we obtain
\begin{equation}
\label{4.30}
 \left \|\int_0^t{\rm e}^{(t-s)\Delta}\,(f(u(s))) ds\right\|_{L^{\infty}(0,\infty; L^p)}\leq CM^m.
\end{equation}
{Note here that we need the condition $(2-m)_+<{N(p-1)\over 2p}=1$ which is satisfied since $m>1.$ Also, we need $a$ to satisfy ${(m-1)p\over p-1}<a<{N(m-1)\over 2}{1\over (2-m)_+},$ which is consistent with the last conditions found on $a,$ since $N= 2p/(p-1).$ }

From \eqref{p} and \eqref{4.30}, it follows that
\begin{equation*}
\begin{split}
\|\Phi(u)\|_{L^{\infty}(0,\infty,\exp L^p)}&\leq \|u_{0}\|_{\exp L^p)}+2CM^m.
\end{split}
\end{equation*}

Now, by  \eqref{med3} the inequality \eqref{4.12} gives
\begin{equation*}
t^{\sigma}\|\Phi(u)\|_{p}\leq \|u_{0}\|_{\exp L^p}+ CM^m.
\end{equation*}
If we choose  $M$ and $\varepsilon$ small then  $\Phi$ maps $Y_M$ into itself.
Moreover, thanks to the inequality \eqref{4.12}  we obtain that  $\Phi$ is a contraction  map on $Y_M$. The conclusion follows by the Banach fixed point theorem.

\subsection{Proof of the statement \eqref{1.8}}

{The proof of the statement \eqref{1.8} is done in \cite{MT-ICM2018}. For reader convenience we recall it here.} Let $q\geq \max(N/2,p)$. Using Lemma \ref{Orl-Leb} Part (iii) that is the embedding $L^p\cap L^\infty\hookrightarrow \EP$, and Proposition \ref{LPLQQ}, we write
\begin{eqnarray}\label{V.549}
  \|u(t)-{\rm e}^{t\Delta}u_0\|_{\EP} &\leq& \int_0^t\|{\rm e}^{(t-s)\Delta}
  f(u(s))\|_{\EP}\,ds\nonumber\\
  &\leq & C\int_0^t\|{\rm e}^{(t-s)\Delta}f(u(s))\|_{p}ds+C\int_0^t\|{\rm e}^{(t-s)\Delta}f(u(s))\|_{{\infty}}\,ds\nonumber\\
  &\leq & C\int_0^t\|f(u(s))\|_{p}ds+C\int_0^t(t-s)^{-\frac{N}{2q}} \|f(u(s))\|_{q}\,ds.
\end{eqnarray}
We now  estimate $\|f(u(s))\|_{r}$ for $r=p,\; q.$  Since
$
    |f(u)|\leq C|u|^m{\rm e}^{\lambda |u|^p},
$
we write
\begin{equation*}
 \|f(u)\|_{r}\leq C \||u|^m({\rm e}^{\lambda |u|^p}-1+1)\|_{r}.
\end{equation*}
Using H\"{o}lder inequality and Lemma \ref{sarah55}, we get
\begin{eqnarray*}
  \|f(u)\|_{r}&\leq& C\|u\|^{m}_{{2mr}}\|{\rm e}^{\lambda |u|^p}-1\|_{{2r}}+\|u\|^{m}_{{mr}}\\
&\leq&C\|u\|^{m}_{\EP}\left(\|{\rm e}^{\lambda |u|^p}-1\|_{{2r}}+1\right).
\end{eqnarray*}
By Lemma \ref{med}, we obtain
\begin{equation}\label{V.559}
\|f(u)\|_{r}\leq C \|u\|^{m}_{\EP} \left((2\lambda r M^p)^{\frac{1}{2r}}+1\right)\leq C \|u\|^{m}_{\EP}.
\end{equation}
Using \eqref{V.559} in \eqref{V.549}, we get
\begin{eqnarray*}
\|u(t)-{\rm e}^{t\Delta}u_0\|_{\EP}&\leq& C\int_0^t \left(\|u(s)\|^{m}_{\EP}+(t-s)^{-\frac{N}{2q}}
\|u(s)\|^{m}_{\EP}\right)ds\\
&\leq& Ct \|u\|^{m}_{L^{\infty}(0,\infty;\,\EP)}+Ct^{1-\frac{N}{2q}}\|u\|^{m}_{L^{\infty}(0,\infty;\,\EP)} \\
&\leq &C_1t+C_2t^{1-\frac{N}{2q}},
\end{eqnarray*}
where $C_1,\,C_2$ are  finite positive constants. Then
$
    \displaystyle\lim_{t\longrightarrow0}\,\|u(t)-{\rm e}^{t\Delta}u_0\|_{\EP}=0,
$
and  statement \eqref{1.8} is now proved. The fact that $u(t)\to u_0$ as $t\to 0$ in the weak$^*$ topology can be done as in \cite{Ioku}. This completes the proof of the theorem.


\section*{Acknowledgments} {The authors wish to thank the anonymous referee for his/her valuable comments which
helped to improve the article.}

\end{document}